\documentclass[preprint]{amsart}

\usepackage{amssymb}
\usepackage{amsthm}
\usepackage[centertags]{amsmath}
\usepackage{amsfonts}
\usepackage{newlfont}
\usepackage{enumerate}

\newtheorem{theorem}{Theorem}

\newtheorem{corollary}{Corollary}
\newtheorem{proposition}{Proposition}
\theoremstyle{definition}
\newtheorem{definition}{Definition}
\newtheorem{example}{Example}

\newtheorem{problem}{Problem}

\newcommand{\intt}{{\mathrm{ int}\,}}
\newcommand{\card}{{\mathrm{ card}\,}}

\newcommand{\clo}{{\mathrm{ cl}\,}}
\newcommand{\conv}{{\mathrm{ conv}}}

\newfont{\gothic}{eufm9 scaled 1400}

\begin{document}


\title[Steinhaus-type property for boundary of a~convex body]{Steinhaus-type property \\ for a boundary of a~convex body}

\author{Wojciech Jab{\l}o{\'n}ski}

\address{Department of Mathematical Modeling, Faculty of Mathematics and Applied Phy\-sics, Rzesz{\'o}w University of Technology, Powsta{\'n}c{\'o}w War\-sza\-wy 6, PL-35-959 Rzesz{\'o}w}

\date{}
\begin{abstract}
We show that if $U\subset\partial A$ is a neighbourhood of a point $x_0\in\partial A$ of the boundary of a convex body $A$ then it has the so-called {\em Stainhaus-type property} ($\intt(U+U)\neq\emptyset$) if and only if $x_0$ is not a point of flatness of the boundary~$\partial A$. This implies that additive functions as well as mid-convex functions, bounded above on~$U$, are continuous.
\end{abstract}

\maketitle

\section{Introduction}

Let $X$ be a real topological vector space.  For nonempty sets $A,B\subset X$ and for $\alpha\in\mathbb{R}$ we define the Minkowiski's operations
$$
\begin{array}{l}
A\pm B:=\{a\pm b\in X: a\in A,\;b\in B\},\\[1ex]
\alpha A:=\{\alpha a\in X:a\in A\,\}.
\end{array}
$$
In the case $B=\{x_0\}\subset X$ we will simply write $x_0+A$ instead of $\{x_0\}+A$. Moreover, for $A\subset X$ and a positive integer $n$ we define
$$
{\mathcal{S}}_n(X)=\underbrace{A+\ldots+A}_{n\text{ times}}.
$$

Let $D\subset X$ be a non-empty convex set (i.e. $\lambda D+(1-\lambda)D\subset D$ for all $\lambda\in[0,1]$). A~function $g:X\to \mathbb{R}$ is called {\em mid-convex} ({\em J-convex}) provided $g \left( \frac{x+y}{2}\right)\leq\frac{g(x)+g(y)}{2}$ for $x,y\in D$. Similarly, a function $a:X\to \mathbb{R}$ is {\em additive} if $a(x+y)=a(x)+a(y)$ for $x,y\in X$. The classical results concerning either J-convex functions or additive functions state that boundedness of such functions on sufficiently large sets imply their continuity. In connection with these results R.~Ger and M.~Kuczma introduced in~\cite{GerKuczma} (for $X=\mathbb{R}^n$) the following classes of sets:
$$
\begin{array}{ll}
{\mathcal{A}}(X)=\{T \subset X: & \mbox{every mid-convex function } f:D\to\mathbb{R} \mbox{ bounded above }\\[.5ex]
& \mbox{on } T \; \mbox{ is continuous on } D,  \mbox{where } \; T\subset D\subset X \\[.5ex]
& \mbox{and } D \mbox{ is nonempty and open}\}, \\[1.5ex]
{\mathcal{B}}(X)=\{T\subset X: &
\mbox{every additive function } a:X\to\mathbb{R} \mbox{ bounded above}\\[.5ex]
& \mbox{on $T$ is continuous}\}.
\end{array}
$$
It is known that every additive function is mid-convex so we always have ${\mathcal{A}}(X)\subset{\mathcal{B}}(X)$. The equality ${\mathcal{A}}(X)={\mathcal{B}}(X)$ holds for $X$ being a real Baire topological vector space (see~\cite{GerKominek}), in particular for $X=\mathbb{R}^n$ (see~\cite{MEK}).

The question which sets belong to either
${\mathcal{A}}(X)$ or ${\mathcal{B}}(X)$ has been a subject of many papers.
From the classical results for mid-convex functions (see Berstein-Doetsch theorem~\cite{BD} and its generalization~\cite{TTZ}) we get $T\in {\mathcal{A}}(X)$ provided $\intt T\neq\emptyset$ in a~real topological vector space. This jointly with the property (see~\cite{Kuczma2})
$$
\mbox{ if }\; {\mathcal{S}}_n(T)\in {\mathcal{A}}(X)\;({\mathcal{S}}_n\in {\mathcal{B}}(X)) \;\mbox{ for some $n\geq2$ then also }\; T\in {\mathcal{A}}(X)\; (T\in {\mathcal{B}}(X)),
$$
implies
$$
T\in {\mathcal{A}}(X)\; (T\in {\mathcal{B}}(X)) \;\mbox{ provided }\, \intt{\mathcal{S}}_n(T)\neq\emptyset \;\mbox{ for some }\,n\geq2.
$$
This nice property leads directly to the Steinhaus-type theorems (theorems of Steinhaus, Piccard and their generalizations) which imply that if $A,B\subset X$ are not small in some sense then $\intt(A+B)\neq\emptyset$. Sets of positive Lebesgue measure, or sets of the second category with the Baire property are in some sense big ones. However there are known "thin sets" $T$ for which the set $T+T$ has a nonempty interior (and these sets belong to ${\mathcal{A}}(X)$ and ${\mathcal{B}}(X)$).

It is known that $C+C=[0,2]$ for the Cantor ternary set, so $C\in{\mathcal{A}}(\mathbb{R})$. M.~Kuczma proved in \cite{Kuczma1} that the graph of a continuous non-affine function defined on an interval belongs to~${\mathcal{A}}(\mathbb{R}^2)$. This results has been next generalized for higher dimensions by R. Ger~\cite{Ger} by proving that the regular $(n-1)$-dimensional hypersurface which is not contained in an $(n-1)$-dimensional affine hyperplane belongs to~${\mathcal{A}}(\mathbb{R}^n)$. In~\cite{Jablonski1} the regularity assumption was weakened and it was proved that the graph of a continuous non-affine function defined on a~non-empty open subset of $\mathbb{R}^{n-1}$ is in~${\mathcal{A}}(\mathbb{R}^n)$. The last results has been generalized  to the product space $X\times\mathbb{R}$ for a~real continuous function defined on an open subset of a real normed space~$X$. R. Ger and M.~Sablik shown in~\cite{GerSablik} that $S(e,\varepsilon)\in{\mathcal{A}}(X)$ for arbitrary neighbourhood $S(e,\varepsilon)\subset S_1$ of an extremal point~$e$ of the unit sphere~$S_1$ in a real normed space~$X$.

P. Volkmann and W. Walter have studied in~\cite{VW} an equivalent condition for an additive function to be continuous. They have proved that an additive operator $F:X\to Y$ mapping a~real normed space $X$ into a real normed space~$Y$ is continuous if and only if $F$ is bounded on a nonempty and (relatively) open subset $A$ of a~unit sphere $S_1$ in $X$ which is not included in a~union of two parallel hyperplanes in $X$. Clearly that result can be simply proved by showing that for such the set~$A$ we have $\intt (A+A)\neq\emptyset$.

Since the mentioned above conditions lead to sets which algebraic sum of some number copies have interior points we introduce the following notions (cf. also \cite{BKR}).

\begin{definition}
Let $X$ be a real topological vector space and fix a positive integer $n\geq2$. We say that a set $A\subset X$ has $n$-{\em Steinhaus-type property} ($A\in{\mathcal{SP}}_n(X)$) whenever
$$
\intt_X\left({\mathcal{S}}_n(A)\right)\neq\emptyset.
$$
We say that a set $A\subset X$ has {\em strong} $n$-{\em Steinhaus-type property} ($A\in{\mathcal{SSP}}_n(X)$) provided
$$
\intt_X\left({\mathcal{S}}_{n-1}(A)\right)=\emptyset\;\mbox{ and }\;\intt_X\left({\mathcal{S}}_n(A)\right)\neq\emptyset.
$$
\end{definition}

Our aim is to study here "small" sets in a real normed space which have $2$-Steinhaus-type property. We prove necessary and sufficient condition for subsets of boundary of a convex body to have $2$-Steinhaus-type property. Finally we discuss some examples and we pose some problems concerning strong $n$-Steinhaus-type property.

From now on $(X,\|\cdot\|)$ will be a real normed space and let $(X^*,\|\cdot\|^*)$ be its dual space. For $v\in X$ let $T_v:X\to X$ be a translation $T_v(x)=x+v$ for $x\in X$. Denote by $\overline{B}$ and~$B$ the closed unit ball and the open unit ball and in~$X$, respectively. Let $S$ and $S^*$ be unit spheres in $X$ and in~$X^*$, respectively.

For arbitrary $x,y\in X$ by $[x,y]\subset X$ we denote the {\em line segment} joining points $x$ and~$y$, i.e. the set $[x,y]=\{(1-t)x+ty:t\in[0,1]\,\}$. By a {\em path} in $X$ we mean every continuous function $\gamma:[0,1]\to X$. We will identify a path $\gamma:[0,1]\to X$ with its image $\Gamma:=\gamma([0,1])=\{\gamma(t)\in X:t\in[0,1]\,\}$. By a {\em hyperplane} in~$X$ we mean a set
$$
H_{x_0^*,c}=\{x\in X:x_0^*(x)=c\,\},
$$
with fixed $x_0^*\in S^*$ and $c\in\mathbb{R}$. Moreover, by
$$
H^-_{x_0^*,c}=\{x\in X:x_0^*(x)<c\,\}\;\mbox{ and } H^+_{x_0^*,c}=\{x\in X:x_0^*(x)>c\,\},
$$
we denote {\em open halfspaces} determined by a hyperplane $H_{x_0^*,c}$. A path $\gamma:[0,1]\to X$ is called {\em plane path} provided $\gamma([0,1])\subset H_{x_0^*,c}$ for some $x_0^*\in S^*$ and $c\in\mathbb{R}$, that is if $x_0^*(\gamma(t))=c$ for $t\in[0,1]$. In this case we call $\gamma$ the $x_0^*$-{\em plane path}.

A nonempty set $A\subset X$ is called {\em convex} provided $[x,y]\subset A$ for every $x,y\in A$. A closed and convex set with interior points we will call a~{\em convex body}. By the well known Hahn-Banach theorem (see also~\cite{M}) we get that for arbitrary convex body $A\subset X$ and for every $x_0\in \partial A$ there exist $x_0^*\in S^*$ and $c\in\mathbb{R}$ such that $x_0^*(x_0)=c$ and $x_0^*(x)\leq c$ for every $x\in A$. This means that through every boundary point of a convex body there passes a
plane supporting the body. In particular, for $\overline{B}\subset X$ and for every $x_0\in S=\partial\overline{B}$ there exists $x_0^*\in S^*$ such that $x_0^*(x_0)=\|x_0\|=1$ and $x_0^*(x)\leq 1$ for every $x\in\overline{B}$. Clearly the existing $x_0^*$ need not be unique, so by $\partial A(x_0)$ ($S(x_0)$, respectively) we denote the set of all $x_0^*\in S^*$ satisfying $x_0^*(x_0)=c$  and $x_0^*(x)\leq c$ for $x\in A$ ($x_0^*(x_0)=1$ and $x_0^*(x)\leq 1$ for $x\in\overline{B}$, respectively). For a~conved body $A\subset X$ a point $x_0\in \partial A$ is called a {\em point of flattening} of~$\partial A$, if there exists $\varepsilon>0$ such that $\partial A\cap(x_0+\varepsilon B)\subset H_{x_0^*,c}$ for some $x_0^*\in \partial A(x_0)$ and $c\in\mathbb{R}$, i.e. if $x_0^*(x)=c$ for $x\in \partial A\cap(x_0+\varepsilon B)$. Note that if $x_0\in \partial A$ is
a flattening point of $\partial A$, then $\card \partial A(x_0)=1$.

The convex hull, closure, interior, boundary and of a set $A\subset X$ will be indicated by $\conv A$, $\clo A$ and $\intt A$. The set $A_\varepsilon=A+\varepsilon B=\{a+\varepsilon b:a\in A,\;b\in B\,\}$ will be an $\varepsilon$-{\em neighrbourhood} of~$A$.

\section{Stainhaus-type results}

We prove here our results for the unit sphere. We begin with the following lemma.

\begin{proposition}\label{lem1}
Let us fix $\varepsilon\in(0,1)$ and $x_0\in S$. Let $\gamma:[0,1]\to X$ be a path, which is not $x_0^*$-plane for some $x_0^*\in S(x_0)$. Then there are $\alpha,\eta>0$ and $t_0\in[0,1]$ such that
\begin{equation}\label{elem1}
(\Gamma-\gamma(t_0)+(1-\alpha)x_0+z) \cap\left(S\cap\left (x_0+\varepsilon B\right)\right) \neq\emptyset\qquad\mbox{ for all }\,z\in\eta B.
\end{equation}
\end{proposition}

\begin{proof}
Let us fix $\varepsilon\in(0,1)$, $x_0\in S$ and consider a neighbourhood $S\cap(x_0+\varepsilon\overline{B})$ in a relative topology on~$S$. If $\gamma$ be a non $x_0^*$-plane path in $X$ for some $x_0^*\in S(x_0)$ then  for
$$
m:=\inf\{(x_0^*\circ\gamma)(t):t\in[0,1]\,\}\;\mbox{ and }\; M:=\sup\{(x_0^*\circ\gamma)(t):t\in[0,1]\,\}
$$
we have $M-m>0$. Moreover, $x_0^*\circ\gamma$ is a continuous function on a compact interval $[0,1]$, so there are $t_1,t_2\in[0,1]$ with $(x_0^*\circ\gamma)(t_1)=m$ and $(x_0^*\circ\gamma)(t_2)=M$. Without loss of generality we may assume that $t_0=0$ and $t_2=1$ (if $t_1>t_2$ we should change the orientation of our path i.e. we shall take the path $\gamma_1:[0,1]\to X$, $\gamma_1(t)=\gamma(1-t)$; we have to take next a restriction $\gamma|_{[t_1,t_2]}$ or $\gamma_1|_{[t_1,t_2]}$). Finally $\gamma$ is uniformly continuous as a continuous function on a compact interval~$[0,1]$. Therefore we find $\delta>0$ such that
\begin{equation}\label{uc}
\|\gamma(s)-\gamma(t)\|<\frac{\varepsilon}{2}\qquad \mbox{ for all }\,s,t\in[0,1]\;\mbox{ such that }\,|s-t|<\delta.
\end{equation}
Let $n\in\mathbb{N}$ be such that $n\delta\geq1$. We show that there exist $t_0,t_1\in[0,1]$ such that $t_0<t_1<t_0+\delta$ and
\begin{eqnarray}
&&\displaystyle \gamma(t)\in \gamma(t_0)+\frac{\varepsilon}{2}B \qquad\qquad\qquad\qquad\qquad \mbox{ for }\, t\in[t_0,t_1],\label{c1}\\
&&\displaystyle (x_0^*\circ\gamma)(t_1)-(x_0^*\circ\gamma)(t_0) \geq\frac{M-m}{n}.\label{c2}
\end{eqnarray}
Indeed, for arbitrary $t_0,t_1\in[0,1]$ with $t_0<t_1$ and $t_1-t_0<\delta$, from~\eqref{uc} we get
$$
\|\gamma(t)-\gamma(t_0)\|<\frac{\varepsilon}{2}\qquad\mbox{ for }\,t\in[t_0,t_1].
$$
This implies~\eqref{c1}. In order to prove~\eqref{c2} let us suppose that contrary to~\eqref{c2} we have
\begin{equation}\label{c3}
(x_0^*\circ\gamma)(t_1)-(x_0^*\circ\gamma)(t_0) <\frac{M-m}{n}
\end{equation}
for all $t_0,t_1\in[0,1]$ such that $t_0<t_1$ and $t_1-t_0<\delta$. Since $n\delta\geq1$ we can find $0=s_0<s_1<\ldots<s_n=1$ such that $s_k-s_{k-1}<\delta$ for $k\in\{1,\ldots,n\}$. Then by~\eqref{c3} we obtain
$$
(x_0^*\circ\gamma)(s_k)-(x_0^*\circ\gamma)(s_{k-1}) <\frac{M-m}{n}\qquad\mbox{ for }\,k\in\{1,\ldots,n\}.
$$
Thus
$$
\begin{array}{rcl}
M-m&=&\displaystyle (x_0^*\circ\gamma)(1)-(x_0^*\circ\gamma)(0)\\[1ex]
&=&\displaystyle \sum_{k=1}^n \left((x_0^*\circ\gamma)(s_k)-(x_0^*\circ\gamma)(s_{k-1})\right)\\[1ex] &<&\displaystyle \sum_{k=1}^n\frac{M-m}{n}=n\cdot\frac{M-m}{n}=M-n.
\end{array}
$$
This contradiction implies that~\eqref{c2} holds true for some $t_0,t_1\in[0,1]$ such that $t_0<t_1<t_0+\delta$.

Put $\alpha:=\min\left(\frac{M-m}{2n},\frac{\varepsilon}{4}\right)$ and define $\widetilde{\gamma}:[0,1]\to X$,
$$
\widetilde{\gamma}(t)=(T_{-\gamma(t_0)+(1-\alpha)x_0}\circ\gamma)(t) =\gamma(t)-\gamma(t_0)+\left(1-\alpha\right) x_0\qquad\mbox{ for }\,t\in[0,1].
$$
Hence $\widetilde{\Gamma}=\Gamma-\gamma(t_0)+\left(1-\alpha\right) x_0$. Then $\widetilde{\gamma}(t_0)=(1-\alpha)x_0$ and $\|\alpha x_0\|\leq\frac{\varepsilon}{4}$, so $\widetilde{\gamma}(t_0)=(1-\alpha)x_0\in x_0+\frac{\varepsilon}{4}\overline{B}$ and $\|\widetilde{\gamma}(t_0)\|=\|(1-\alpha)x_0\|<1$. Moreover, $(x_0^*\circ\widetilde{\gamma})(t_0)= (1-\alpha)x_0^*(x_0)=1-\alpha<1$. Furthermore, from~\eqref{c1} for every $t\in[t_0,t_1]$ we get
$$
\widetilde{\gamma}(t)=\gamma(t)-\gamma(t_0)+(1-\alpha)x_0 \in\gamma(t_0)-\gamma(t_0)+\frac{\varepsilon}{2}B +(1-\alpha)x_0\subset x_0+\frac{3\varepsilon}{4}B,
$$
which means that $\widetilde{\gamma}([t_0,t_1])\subset x_0+\frac{3\varepsilon}{4}B$. Finally, by~\eqref{c2} we have
$$
\begin{array}{rcl}
(x_0^*\circ\widetilde{\gamma})(t_1)&=& x_0^*(\gamma(t_1)-\gamma(t_0)+(1-\alpha)x_0)\\[2ex]
&=& (x_0^*\circ\gamma)(t_1)-(x_0^*\circ\gamma)(t_0)+ (1-\alpha)x_0^*(x_0)\\[1ex]
&\geq& \displaystyle \frac{M-m}{n}+\left(1-\frac{M-m}{2n}\right)=1+\frac{M-m}{2n}>1,
\end{array}
$$
so $\|\widetilde{\gamma}(t_1)\|>1$.

We have thus proved that $\widetilde{\gamma}(t_0)$ is an interior point of the set $(x_0+\frac{\varepsilon}{2}B)\cap B$, $\widetilde{\gamma}([t_0,t_1])\subset x_0+\frac{3\varepsilon}{4}B$ and $\widetilde{\gamma}(t_1)$ is an interior point of the set $(x_0+\frac{3\varepsilon}{4}B) \cap H^+_{x_0^*,1}$. Fix $\eta<\alpha$ such that $\widetilde{\gamma}(t_0)+\eta B\subset (x_0+\frac{\varepsilon}{2}B)\cap B$ and $\widetilde{\gamma}(t_1)+\eta B\subset (x_0+\frac{3\varepsilon}{4}B)\cap H^+_{x_0^*,1}$.

For arbitrary $z\in\eta B$ we define $\widehat{\gamma}_z:[t_0,t_1]\to X$ by
$$
\widehat{\gamma}_z(t)=(T_z\circ\widetilde{\gamma})(z)= \gamma(t)-\gamma(t_0)+(1-\alpha)x_0+z\qquad\mbox{ for }\,t\in[t_0,t_1],
$$
and let $\widehat{\Gamma}_z:=\{\widehat{\gamma}_z(t):t\in[t_0,t_1]\,\}$. For every $t\in[t_0,t_1]$ we have
$$
\widehat{\gamma}_z(t)=(T_z\circ\widetilde{\gamma})(t)= \widetilde{\gamma}(t)+z\in x_0+\frac{3\varepsilon}{4}B+\eta B\subset x_0+\frac{3\varepsilon}{4}B+\frac{\varepsilon}{4}B=x_0+\varepsilon B,
$$
hence $\widetilde{\Gamma}_z\subset x_0+\varepsilon B$ for every $z\in\eta B$. Moreover,
$$
\widehat{\gamma}_z(t_0)=\widetilde{\gamma}(t_0)+z\in\widetilde{\gamma}(t_0)+\eta B\subset\left(x_0+\frac{\varepsilon}{2}B\right)\cap B\subset \left(x_0+\varepsilon B\right)\cap B.
$$
Finally, $\|x_0^*\|^*=1$ and $\eta<\alpha\leq\frac{M-m}{2n}$, so $\|z\|<\frac{M-m}{2n}$,
$$
\begin{array}{rcl}
(x_0^*\circ\widehat{\gamma}_z)(t_1)&=&x_0^*( \gamma(t_1)-\gamma(t_0)+(1-\alpha)x_0+z)\\[2ex]
&=& x_0^*(\gamma(t_1))-x_0^*(\gamma(t_0))+(1-\alpha)x_0^*(x_0) -x_0^*(-z)\\[1ex]
&\geq&\displaystyle \frac{M-m}{n}+\left(1-\frac{M-m}{2n}\right)- \|x_0^*\|\cdot\|z\|\\[1ex]
&>&\displaystyle 1+\frac{M-m}{2n}-\frac{M-m}{2n}=1,
\end{array}
$$
and
$$
\widehat{\gamma}_z(t_1)=\widetilde{\gamma}(t_1)+z\in\widetilde{\gamma}(t_1)+\eta B\subset\left(x_0+\frac{3\varepsilon}{4}B\right)\cap H^+_{x_0^*,1}\subset \left(x_0+\varepsilon B\right)\cap H^+_{x_0^*,1}.
$$
Let $|\widehat{\gamma}_z|:[t_0,t_1]\to\mathbb{R}$ be defined by $$
|\widehat{\gamma}_z|(t)=\|\widehat{\gamma}_z(t)\|\qquad\mbox{ for }\,t\in[t_0,t_1].
$$
Since $\widehat{\gamma}_z(t_0)\in B$, so $|\widehat{\gamma}_z|(t_0)<1$. Similarly, $\widehat{\gamma}_z(t_1)\in H^+_{x_0^*,1}$, so $\widehat{\gamma}_z(t_1)\in X\setminus\overline{B}$, hence $|\widehat{\gamma}_z(t_1)|(t_1)>1$. By continuity of $|\widehat{\gamma}_z|$ there exists $t_z\in[t_0,t_1]$ such that $|\widehat{\gamma}_z|(t_z)=\|\widehat{\gamma}_z(t_z)\|=1$ and, moreover, $\widehat{\gamma}_z(t_z)\in(x_0+\varepsilon B)$. Thus $$
\widehat{\Gamma}_z\cap(S\cap(x_0+\varepsilon B))\neq\emptyset \qquad\mbox{ for each }\,z\in\eta B.
$$
This finishes the proof.
\end{proof}

From proposition we derive the following result.

\begin{theorem}\label{th1}
Let us fix $\varepsilon\in(0,1)$, $x_0\in S$ and let $\gamma:[0,1]\to X$ be a path, which is not $x_0^*$-plane for some $x_0^*\in S(x_0)$. Then $\intt((S\cap(x_0+\varepsilon B))\pm\Gamma)\neq\emptyset$.
\end{theorem}

\begin{proof}
Fix $\varepsilon\in(0,1)$, $x_0\in S$ and assume that $\gamma:[0,1]\to X$ is a path, which is not $x_0^*$-plane for some $x_0^*\in S(x_0)$. By Proposition~\ref{lem1} there exist $\alpha,\eta>0$ and $t_0\in[0,1]$ such that~\eqref{elem1} holds. Hence for arbitrary $z\in\eta B$ we find $a\in S\cap(x_0+\varepsilon B)$ such that $a\in \Gamma-\gamma(t_0)+(1-\alpha)x_0+z$. This implies that $a=b-\gamma(t_0)+(1-\alpha)x_0+z$ for some $b\in\Gamma$. Thus
$$
-\gamma(t_0)+(1-\alpha)x_0+\eta B\ni -\gamma(t_0)+(1-\alpha)x_0+z=a-b\in (S\cap(x_0+\varepsilon B))-\Gamma,
$$
so $\intt((S\cap(x_0+\varepsilon B))-\Gamma)\neq\emptyset$.

To prove $\intt((S\cap(x_0+\varepsilon B))+\Gamma)\neq\emptyset$ it is enough to apply Proposition~\ref{lem1} for the path~$-\Gamma$. Indeed, if $\Gamma$ is not $x_0^*$-plane, then also $-\Gamma$ is not $x_0^*$-plane. Then for arbitrary $z\in\eta B$ we find $a'\in S\cap(x_0+\varepsilon B)$ such that $a'\in -\Gamma+\gamma(t_0)+(1-\alpha)x_0+z$. Thus $a'=-b'+\gamma(t_0)+(1-\alpha)x_0+z$ for some $b'\in\Gamma$ and $$
\gamma(t_0)+(1-\alpha)x_0+\eta B\ni \gamma(t_0)+(1-\alpha)x_0+z=a+b\in (S\cap(x_0+\varepsilon B))+\Gamma,
$$
so $\gamma(t_0)+(1-\alpha)x_0+\eta B\subset (S\cap(x_0+\varepsilon B))+\Gamma$ and $\intt((S\cap(x_0+\varepsilon B))+\Gamma)\neq\emptyset$.
\end{proof}

From Theorem! \ref{th1} we obtain the following crucial corollary.

\begin{corollary}\label{cor1}
Fix $x_0\in S$. Then $x_0$ is not a flattening point of $S$ if and only if
$$
\intt((S\cap(x_0+\varepsilon B))+(S\cap(x_0+\varepsilon B)))\neq\emptyset \qquad\mbox{ for arbitrary small }\, \varepsilon>0.
$$
\end{corollary}

\begin{proof}
Let us fix $x_0\in S$. Assume first that $x_0$ is a flattening point of $S$. Then there exist $\varepsilon>0$ and $x_0^*\in S(x_0)$ such that $x_0^*(x)=1$ for all $x\in S\cap(x_0+\varepsilon B)$. Thus
$$
x_0^*(x_1+x_2)=x_0^*(x_1)+x_0^*(x_2)=2\qquad\mbox{ for all }\,x_1,x_2\in S\cap(x_0+\varepsilon B),
$$
so $\intt((S\cap(x_0+\varepsilon B))+(S\cap(x_0+\varepsilon B)))=\emptyset$.

Let us assume now that $x_0$ is not a flattening point of $S$. Then for all $x_0^*\in S(x_0)$ and arbitrary small $\varepsilon>0$ there exists $z_0\in S\cap(x_0+\varepsilon B)$ such that $x_0^*(z_0)<1$. Let us fix $x_0^*\in S(x_0)$, $\varepsilon\in(0,1)$ and $z_0\in S\cap(x_0+\varepsilon B)$ with $x_0^*(z_0)<1$. We consider a path $\gamma:[0,1]\to S$,
$$
\gamma(t)=\|(1-t)x_0+tz_0\|\qquad\mbox{ for }\,t\in[0,1].
$$
Then $\gamma$ is not $x_0^*$-plane, and by Theorem~\ref{th1} we get
$$
\intt((S\cap(x_0+\varepsilon B))+(S\cap(x_0+\varepsilon B))) \supset \intt((S\cap(x_0+\varepsilon B))+\Gamma)\neq\emptyset.
$$
This finishes the proof.
\end{proof}

From Corollary~\ref{cor1} one can easily derive the result proved in~\cite{GerSablik}. Moreover, it is clear, that the result proved above is invariant with respect to translations $X\ni x\mapsto x+v\in X$ and uniform scalings $X\ni x\mapsto\lambda x\in X$ with fixed $v\in X$ and $\lambda\neq0$. Thus we have the following result.

\begin{corollary}
For every $\delta>0$ and for all $x\in X$, $r>0$ we have
$$
x+r{\mathcal{S}}_1(x_0,x_0^*,\delta)\in {\mathcal{A}}(X)\cap{\mathcal{B}}(X).
$$
\end{corollary}

Using Theorem~\ref{th1} and Corollary~\ref{cor1} we get also another proof or theorem by P.~Volkmann and W.~Walter~\cite{VW}.

\begin{corollary}
Let $X$ be a real normed space of dimension at least~$2$. If $A\subset S_1$ is a~nonempty and (relatively) open subset of a~unit sphere in $X$ which is not included in a~union of two parallel hyperplanes in $X$, then $\intt(A+A)\neq\emptyset$ and hence $A\in{\mathcal{A}}(X)\cap{\mathcal{B}}(X)$.
\end{corollary}

\begin{proof}
If $A\subset S_1$ is a~nonempty and (relatively) open subset of a~unit sphere in $X$ which is not included in a~union of two parallel hyperplanes then we have the following two possibilities:
\begin{enumerate}[(a)]
\item $A$ contains a point and its relatively open in $S_1$ neighbourhood which not $x^*$-flat for every $x^*\in S_1^*$,
\item $A$ contains two relatively open components contained in two different non-paralel hyperplanes.
\end{enumerate}
In the first case we use Corollary~\ref{cor1} to obtain our statement. In the second case we take any relatively open $x^*$-plane subset of the first component and arbitrary line segment contained in the second one which is clearly not $x^*$-plane. Then Theorem~\ref{th1} finishes the proof.
\end{proof}

We are in position to discuss now the general case of a convex body in~$X$. We prove

\begin{theorem}
Let $K\subset X$ be a convex body and fix $x_0\in \partial K$. Then $x_0$ is not a~flattening point of $\partial K$ if and only if
$$
\intt((\partial K\cap(x_0+\varepsilon B))+(\partial K\cap(x_0+\varepsilon B)))\neq\emptyset \qquad\mbox{ for arbitrary small }\, \varepsilon>0.
$$
\end{theorem}

\begin{proof}
Let $\|\cdot\|$ be a norm in~$X$. Fix $x_0\in\partial K$ and let $z_0\in\intt K$. Since the property which we want to prove is invariant with respect to translations without loss of generality we can assume that $z_0=0$. Let $\delta:=\|x_0\|$ and consider $A:=\partial K\cap\left(x_0+\frac{\delta}{2}B\right)$. Then the set $V:=\clo\conv\left(A\cup\frac{\delta}{4}B\cup(-A)\right)$ is closed, convex, bounded, symmetric with respect to~$0$ and contains $0$ in its interior. Thus the Minkowski's functional $\mu_V$ properly defines a norm $\|\cdot\|_V$ equivalent to $\|\cdot\|$ in~$X$. Obviously dual spaces for $(X,\|\cdot\|)$ and $(X,\|\cdot\|_V)$ coincides. Let $S_V$ and $B_V$ be  unit sphere and unit ball in the norm~$\|\cdot\|_V$, respectively. Then for sufficiently small $\varepsilon>0$ we have
$$
S_V\cap(x_0+\varepsilon B_V)\subset\partial K\cap\left(x_0+\frac{\delta}{2}B\right).
$$
Then from Corollary~\ref{cor1} we get our statement.
\end{proof}

\section{Examples and problems}

We begin with the following simple and obvious example.

\begin{example}
For every $n\in\mathbb{N}$ with $n\geq2$ in the space $\mathbb{R}^n$ there exists a set $A\in{\mathcal{SSP}}_n(\mathbb{R}^n)$. If $e_i$ for $i\in\{0,1,\ldots,n\}$ is a canonical affine basis of~$\mathbb{R}^n$, then the polyline joining points $e_0$, $e_1$, \ldots, $e_n$ has the desired property.
\end{example}

Based on the above example we can ask the following question.

\begin{problem}
Fix $n\in\mathbb{N}$ with $n\geq2$.
Let $\gamma:[0,1]\to\mathbb{R}^n$ define a continuous curve~$\Gamma$ which goes through $(n+1)$ affinely independent points in $\mathbb{R}^n$. Is it true that $\Gamma\in{\mathcal{SP}}_n(\mathbb{R}^n)$?
\end{problem}

Using Theorem 1.3 (a) and (e) from~\cite{BBFS} we get the next example.

\begin{example}
Let $C_{\lambda}$ be a Cantor-type set with selfsimilarity ratio $\lambda$. For $k\in\mathbb{N}$ with $k\geq2$ and for $\lambda\in\left[\frac{1}{k+1},\frac{1}{k}\right)$ we have $C_{\lambda}\in{\mathcal{SSP}}_k(\mathbb{R})$ that is
$$
\intt\left({\mathcal{S}}_{k-1}(C_{\lambda})\right)=\emptyset
\qquad\mbox{ and }\qquad{\mathcal{S}}_k(C_{\lambda})=[0,k].
$$
\end{example}

We can generalize this example to higher dimensions.

\begin{example}
Let $C_{\lambda}$ be a Cantor-type set with selfsimilarity ratio $\lambda$. For $k,n\in\mathbb{N}$ with $k,n\geq2$ and for $\lambda\in\left[\frac{1}{n+1},\frac{1}{n}\right)$ we have $C_{\lambda}\times[0,1]^{n-1}\in{\mathcal{SSP}}_k(\mathbb{R}^n)$ that is
$$
\intt\left({\mathcal{S}}_{k-1}(C_{\lambda})\right)=\emptyset
\qquad\mbox{ and }\qquad{\mathcal{S}}_k(C_{\lambda})=[0,k]^n.
$$
\end{example}

\end{document}